\documentclass{amsart}
\newcommand{\dd}

\usepackage{amsmath}
\usepackage[utf8]{inputenc}

\newcommand{\be}{\begin{equation}}
\newcommand{\ee}{\end{equation}}
\newcommand{\bea}{\begin{eqnarray}}
\newcommand{\eea}{\end{eqnarray}}

\newcommand{\bee}{\begin{eqnarray*}}
\newcommand{\eee}{\end{eqnarray*}}

 \newtheorem{thm}{Theorem}[section]
 \newtheorem{cor}[thm]{Corollary}
 
 \newtheorem{lem}[thm]{Lemma}
 \newtheorem{prop}[thm]{Proposition}
 \theoremstyle{definition}
 
 \newtheorem{fac}[thm]{Fact}
 \theoremstyle{remark}
 
 \theoremstyle{example}
 
\usepackage{enumitem}
 \usepackage{color}

\begin{document}

\title [GENERALIZED SKEW DERIVATION ON IDEAL WITH ENGEL
CONDITIONS ]
{GENERALIZED SKEW DERIVATION ON IDEAL WITH ENGEL
	CONDITIONS}

\author[A. Pandey, B. Prajapati]{ Ashutosh Pandey, Balchand Prajapati}
\address{A. Pandey, School of Liberal Studies, Ambedkar University Delhi, Delhi-110006, INDIA.}
\email{ashutoshpandey064@gmail.com}
\address{B. Prajapati, School of Liberal Studies, Ambedkar University Delhi, Delhi-110006, INDIA.}
\email{balchand@aud.ac.in}



\begin{abstract}
Let $R$ be a prime ring of characteristic different from $2$, $U$ be the Utumi quotient ring of $R$ and $C$ be the extended centroid of $R$. Let $F$ be a generalized skew derivation on $R$, $I$ be a non-zero ideal of $R$ and $m, n_1,n_2, \ldots,n_k \geq 1$ are fixed integers such that $[F(u^m),u^{n_1},u^{n_2},\ldots,u^{n_k}]=0$ for all $u\in I$ then there exists $\lambda \in C$ such that $F(x)=\lambda x$ for all $x\in R$.
\end{abstract}

\maketitle
\section{Introduction} Throughout the article $R$ denotes a prime ring with center $Z(R)$. The Utumi quotient ring of $R$ is denoted by $U$. The center of $U$ is called the extended centroid of $R$ and it is denoted by $C$. The definition and construction of $U$ can be found in \cite{beidar1995}. The commutator $ab-ba$ of two elements $a$ and $b$ of $R$ is denoted by $[a,b]$. Define $[a,b]_0=a$ and for $k\geq 1$ the $k$th commutator of $a$ and $b$ is defined as $[a,b]_k= [[a,b],b]_{k-1}=\sum_{i=0}^{k}(-1)^{i}{k\choose i}b^{i}ab^{k-i}$. Also $[a_1,a_2,\ldots,a_k]=[[a_1,a_2,\ldots,a_{k-1}],a_k]$ for all $a_1,a_2,\ldots,a_k \in R$, and for $k\geq 2$. An additive mapping $d:R\rightarrow R$ is said to be a derivation if $d(xy)=d(x)y +xd(y)$ for all $x, y\in R$. An additive mapping $F:R\rightarrow R$ is said to be a generalized derivation if there exists a derivation $d$ on $R$ such that $F(xy)=F(x)y+xd(y)$ for all $x,y\in R$. In \cite{posner1957} Posner proved that if $d$ is derivation of a prime ring $R$ such that $[d(x),x]\in Z(R)$ for all $x\in R$ then either $d=0$ or $R$ is a commutative ring. In \cite{lanski1993engel} Lanski generalized the Posner's result by proving it on Lie ideal $L$ of $R$. More precisely, Lanski proved that if $[d(x),x]_{k}\in Z(R)$ for all $x\in L$ and $k\geq 1$ then char$(R)=2$ and $R\subseteq M_2(\mathbb{F})$, for a field $\mathbb{F}$, equivalently $R$ satisfies standard identity $s_4$. In 2008,  Arga\c{c} et al. in \cite{argac2008}, generalized Lanski's result by replacing derivation $d$ by generalized derivation $F$. More precisely they proved that $[F(x),x]_k=0$, for all $x\in L$, then either $F(x)=ax$ with $a\in C$ or $R$ satisfies the standard identity $s_4$. The study of generalized derivations on Lie ideals and on left ideals are given in \cite{dhara2011,beidar1995,de2012,argac2008} where further references can be found out. More recently in \cite{dhara2016eng} Dhara\c{c} et al. proved the following:\\
\textit{Let $R$ be a prime ring with its Utumi ring of quotients $U$, $G$ a nonzero generalized
derivation of $R$ and $L$ a noncentral Lie ideal of $R$. Suppose that $[G(x^{m} ), x^{n_2}, \ldots,x^{n_k}] = 0$ for all $x \in L$,where $m,n_1, n_2, \ldots, n_k \geq 1$ are fixed integers. Then one of the following holds:
\begin{enumerate}
\item there exists $\beta \in C$ such that $G(x) = \beta x$ for all $x \in R$.
\item $R$ satisfies the standard identity $s_4$.
\end{enumerate}}
In this article we continue this line of investigation concerning the identity $[F(u^m),u^{n_1},u^{n_2},\ldots,u^{n_k}]=0$ for all $u\in R$, where $m,n_1,n_2, \ldots,n_k\geq 1$ are fixed integers and $F$ is a generalized skew derivation. More precisely we shall prove the following:

\textbf{Main Theorem:} Let $R$ be a prime ring of characteristic different from $2$, $U$ be the Utumi quotient ring of $R$ and $C$ be the extended centroid of $R$. Let $F$ be a generalized skew derivation on $R$, $I$ be a two sided ideal of $R$ and $m, n_1,n_2, \ldots,n_k \geq 1$ are fixed integers such that $[F(u^m),u^{n_1},u^{n_2},\ldots,u^{n_k}]=0$ for all $u\in I$ then there exists $\lambda \in C$ such that $F(x)=\lambda x$ for all $x\in R$.

We recall the following facts that are useful to prove our main theorem:
\begin{fac}\label{fac4.1}
Let $f(x_i,d(x_i),\alpha(x_i))$ is a generalized polynomial identity for a prime ring $R$, $d$ is a outer skew derivation and $\alpha$ is outer automorphism of $R$ then $R$ also satisfies the generalized polynomial identity $f(x_i,y_i,z_i)$, where $x_i,y_i,z_i$ are distinct indeterminates. (\cite[Theorem 1]{Chuang2005})
\end{fac}

\begin{fac}(\cite[Theorem 6.5.9]{Kharchenko1991})\label{fac4.2}
Let $R$ be a prime ring satisfies polynomial identity of the type $f(x_j^{\alpha_i\triangle_k})=0$, where $f(z_j^{(i,k)})$ is generalized polynomial identity with coefficient from $U$, $\triangle_1,\ldots,\triangle_n$ are mutually different correct words from a reduced set of skew derivations commuting with all the corresponding automorphisms and $\alpha_1,\ldots,\alpha_m$ are mutually outer automorphisms. In this case the identity $f(z_j^{(i,k)})=0$ is valid for $U$.
\end{fac}
\begin{fac}\label{fac1}
Let $K$ be an infinite field and $m\geq 2$ an integer. If $P_1,\ldots,P_k$ are non-scalar matrices in $M_m(K)$ then there exists some invertible matrix $P \in M_m(K)$ such that each matrix $PP_1P^{-1},\ldots,PP_kP^{-1}$ has all non-zero entries. \cite{de2012}
\end{fac}

\begin{fac}\label{fac2}
Let $K$ be any field and $R=M_m(K)$ be the algebra of all $m\times m $ matrices over $K$ with $m\geq 2$. Then the matrix unit $e_{ij}$ is an element of $[R, R]$ for all $1\leq i\neq j\leq m$.
\end{fac}

\begin{fac}\label{fac3}
Every generalized skew derivation $F$ of $R$ can be uniquely extended to a generalized derivation of $U$ and its assume the form $F(x)=ax+d(x)$, for some $a\in U$ and a skew derivation $d$ on $U$  \cite{chuang1993differential}.
\end{fac}

\begin{fac}\label{fact4}
If $I$ is a two-sided ideal of $R$, then $R$, $I$ and $U$ satisfy the same differential identities \cite{lee1992semiprime}.
\end{fac}

\begin{fac}\label{fact5}
If $I$ is a two-sided ideal of $R$, then $R$, $I$ and $U$ satisfies the same generalized polynomial identities with coefficients in $U$ (\cite{beidar1978ri}).  Further $R$, $I$ and $U$ satisfy the same generalized polynomial identities with automorphism in $U$  \cite{chuang1993differential}.
\end{fac}

\begin{fac}(Kharchenko [Theorem 2,\cite{kharchenko1978diff}]\label{fact6}
Let $R$ be a prime ring, $d$ a non zero derivation on $R$ and $I$ a non zero ideal of $R$. If $I$ satisfies the differential identity
$$f(r_1,\ldots,r_n,d(r_1),\ldots,d(r_n))=0$$ for all $r_1,\ldots,r_n \in I$, then either
\begin{itemize}
\item [(i)] $I$ satisfies the generalized polynomial identity $f(r_1,\ldots,r_n,x_1,\ldots,x_n)=0$
\[\text{or}\]
\item [(ii)] $d$ is $U$-inner i.e., for some $q \in U, d(x)=[q,x]$ and $I$ satisfies the generalized polynomial identity
 $f(r_1,\ldots,r_n,[q,r_1],\ldots,[q,r_n])=0.$
\end{itemize}
\end{fac}

\begin{fac} \big(Theorem 4.2.1 (Jacobson density theorem)\cite{beidar1995}\big)\label{fact7}
Let $R$ be a primitive ring with $V_R$ a faithful irreducible $R$-module and $D = End(V_R)$, then for any positive integer $n$ if $v_1,v_2,\ldots,v_n$ are $D$-independent in $V$ and $w_1,w_2,\ldots,w_n$ are arbitrary in $V$ then there exists $r\in R$ such that $v_ir = w_i$ for $i=1,2,\ldots,n$.
\end{fac}
\begin{fac}\label{fac1.8}
	Let $X = \{x_1, x_2, \ldots \}$ be a countable set consisting of noncommuting indeterminates $x_1, x_2,\ldots$. Let $C\{X\}$ be the free algebra over $C$ on the set $X$. We denote $T = U *_{C} C\{X\}$, the free product of the $C$-algebras $U$ and C\{X\}. The elements of $T $ are called the generalized polynomials with coefficients in $U$. Let $B$ be a set of $C$-independent vectors of $U$. Then any element $f \in T$ can be represented in the form $f=\sum_{i} a_in_i$, where $a_i \in C$ and $n_i$ are $B$-monomials of the form $p_0u_1p_1u_2p_2\cdot u_np_n$, with $p_0, p_1,\ldots, p_n \in B$ and $u_1, u_2,\ldots,u_n \in X$. Any generalized polynomial $f = \sum_{i} a_in_i$ is trivial, i.e., zero element in $T$ if and only if $a_i = 0$ for each $i$. For further details we refer the reader to \cite{chuang1988gpis}.
\end{fac}

  	We begin with the following Proposition:
\begin{prop}\label{prop1}
Let $R$ be a prime ring of characteristic different from $2$, $U$ be the Utumi quotient ring of $R$, $C$ be the extended centroid of $R$ and $\beta \in Aut(U)$. Let $a,b \in U$  and $m, n_1,n_2, \ldots,n_k \geq 1$ are fixed integers such that
\begin{equation} \label{eq1}
 [au^{m} +\beta(u^{m})b,u^{n_1},u^{n_2},\ldots,u^{n_k}]=0\end{equation} for all $u\in R$ then either $\beta$ is identity map on $R$ and $a,b \in C$ or there exists an invertible element $p$ such that $\beta(x) = pxp^{-1}$, for all $x \in R$ with $p^{-1}b \in C$ and $a +b \in C$.
\end{prop}

We need to prove the following lemmas to prove Proposition (\ref{prop1}):

\begin{lem}\label{lem1}
	Let $R$ be a prime ring of characteristic different from $2$, $U$ be the Utumi quotient ring of $R$ and $C$ be the extended centroid of $R$. Let $a,b \in U$  and $m, n_1,n_2, \ldots,n_k \geq 1$ are fixed integers such that
	\begin{equation}\label{eq2}
	[au^{m} +pu^{m}p^{-1}b,u^{n_1},u^{n_2},\ldots,u^{n_k}]=0\end{equation} for all $u \in R$ then $ p^{-1}b \in C$ and $ a+b \in C$.
\end{lem}
\begin{proof}
First assume that $R$ does not satisfy any non-trivial generalized polynomial identity. Let $T = U * C\{u\}	$, the free product of $U$ and $C\{u\}$, $C$-algebra in single indeterminate $u$. Then equation (\ref{eq2}) is a GPI in $T$. If  $p^{-1}b \notin C$ then $p^{-1}b$ and 1 are linearly independent over $C$. Thus from Fact (\ref{fac1.8}), equation (\ref{eq2}) implies 
\begin{equation*}
 u^{n_1+ n_2 +\ldots +n_k}pu^{m}p^{-1}b =0
\end{equation*} in $T$ implying $p^{-1}b =0$, a contradiction. Therefore we conclude that $p^{-1}b \in C$ and hence equation (\ref{eq2}) reduces to :
	\begin{equation}\label{eq3}
	[au^{m} +bu^{m},u^{n_1},u^{n_2},\ldots,u^{n_k}]=0\end{equation}
 again by \cite{dhara2016eng}, equation (\ref{eq3}) implies that $a+b \in C$.\\
Now we consider the case when equation (\ref{eq2}) is a nontrivial polynomial identity for $R$. Since $R$ and $U$ satisfy the same generalized polynomial identities (see Fact \ref{fact5}). Therefore $U$ satisfies equation (\ref{eq2}). In case $C$ is infinite the generalized polynomial identity (\ref{eq2}) is also satisfied by $U {\otimes}_C \bar C$ where $\bar C$ is the algebraic closure of $C$. Since both $U$ and $U {\otimes}_C \bar C$ are prime and centrally closed \cite{erickson1975pr}, we may replace $R$ by $U$ or $U {\otimes}_C \bar C$ according as $C$ is infinite or finite. Thus we may assume that $R$ is centrally closed over $C$ which is either finite or algebraically closed such that $[au^{m} +pu^{m}p^{-1}b,u^{n_1},u^{n_2},\ldots,u^{n_k}]=0$ for all $u \in R$. By Martindale's result \cite{martindale1969pr}, $R$ is a primitive ring with non-zero socle $H$ and $eHe$ is a simple central algebra finite dimensional over $C$, for any minimal idempotent element $e \in R$. Thus there exists a vector space $V$ over a division ring $D$ such that $R$ is isomorphic to a dense subring of ring of $D$-linear transformations of $V$. Since $C$ is either finite or algebraically closed, $D$ must coincide with $C$. \\
 Assume first that $dim_{C}V \geq 3$. If $p^{-1}b \notin C$ then there exists $v \in V$ such that $\{p^{-1}bv,v\}$ is linearly $C$-independent. Since $dim_{C}V \geq 3$ there exists $w \in V$ such that $\{p^{-1}bv,v,w\}$ is linearly $C$-independent. By Jacobson's theorem (see Fact \ref{fact7}) there exists $x \in R$ such that :
\begin{equation*}
xv =0, xp^{-1}bv = p^{-1}bv
\end{equation*} 
Then, $0= [ax^{m} +px^{m}p^{-1}b,x^{n_1},x^{n_2},\ldots,x^{n_k}]v=bv$, a contradiction, because if $ bv =0$, then  $\{p^{-1}bv,v,w\}$ will be $C$-dependent. Thus $\{p^{-1}bv,v\}$ is linearly $C$-dependent therefore $p^{-1}b \in C$ and hence equation (\ref{eq2}) reduces to 
 \begin{equation*}
 	[au^{m} +bu^{m},u^{n_1},u^{n_2},\ldots,u^{n_k}]=0
 \end{equation*}
which implies $a+b \in C$ by \cite{dhara2016eng}.\\
Now if $dim_C V= 2$, then  $ U \cong M_2(C)$. Denote $p = \sum_{ij}e_{ij}p_{ij}, q =p^{-1}b =\sum_{ij}e_{ij}q_{ij} \in M_2 (C)$, for $p_{ij}, q_{ij}\in C$ and $1 \leq i,j \leq2$, where $e_{ij}$ is the usual matrix unit with 1 at $(i,j)^{th}$ place and zero elsewhere. Assume $q \notin C$ then by Fact (\ref{fac1}), all the entries in $q$ is non-zero i.e. $q_{ij}\neq 0$ for $1\leq i,j\leq 2$.\\
Choosing $ u= e_{11}$ in equation $(\ref{eq2})$ and right multiplying by $e_{22}$ we get:
\begin{equation}\label{eqC}
p_{11}q_{12}=0
\end{equation} implying $p_{11}=0$. Let $\phi$ be an automorphism of $U$ then 
\begin{equation}
[\phi(a)u^{m} +\phi(p)u^{m}\phi(p^{-1}b),u^{n_1},u^{n_2},\ldots,u^{n_k}]=0	
\end{equation} is also an identity of $U$. Thus $ \phi(p), \phi(q)$and $\phi(a)$ must satisfy equation (\ref{eqC}). Denote $ \phi(p) = \sum_{ij}e_{ij}p^{\prime}_{ij}, \phi(q) =\sum_{ij}e_{ij}q^{\prime}_{ij}$, for $p^{\prime}_{ij}, q^{\prime}_{ij}\in C$ and $ 1\leq i,j\leq 2$, then from equation (\ref{eqC}), we have $p^{\prime}_{11}q^{\prime}_{12} =0$. In particular chossing $ \phi (u)= (1+ e_{21})u (1-e_{21})$ we get $ p_{12}q_{12}=0$, which implies  $p_{12}=0$. Thus $1^{st}$ row of $p$ is zero, which is a contradiction because $p$ is invertible. Thus $q_{12}=0$, a contradiction. Therefore $q = p^{-1}b \in C$.
Since $ q \in C$ therefore equation (\ref{eq2}) reduces to \begin{equation}\label{eqD}
	[au^{m} +bu^{m},u^{n_1},u^{n_2},\ldots,u^{n_k}]=0
\end{equation} Denote $ c= a+b = \sum_{ij}c_{ij}e_{ij}$, for $c_{ij}\in C$ and $ 1\leq i,j \leq 2$. Suppose $ c\notin C$ then by Fact (\ref{fac1}), all the entries of $c$ is  non-zero i.e. $c_{ij}\neq 0$. Choosing $ u = e_{22}$ in equation (\ref{eqD}) and right multiply by $e_{11}$ we get:
\begin{equation*}
	c_{12}e_{12}=0
\end{equation*} which implies $c_{12}=0$, a contradiction. Therefore $c= a+b \in C$.

\end{proof}

\begin{lem}\label{lem2}
Let $R$ be a prime ring of characteristic different from $2$, $U$ be the Utumi quotient ring of $R$ and $C$ be the extended centroid of $R$. Let $a,b \in U$, $\beta$ be an outer automorphism on $U$,  and $m, n_1,n_2, \ldots,n_k \geq 1$ are fixed integers such that
\begin{equation}\label{eq4}
[au^{m} + \beta(u^m)b,u^{n_1},u^{n_2},\ldots,u^{n_k}]=0\end{equation} for all $u \in R$ then $\beta$ is the identity map on $R$ and $a+b \in C$ unless $b=0$ and $a \in C$.
\end{lem}
\begin{proof}
Since $R$ and $U$ satisfy the same generalized polynomial identity with automorphisms (see Fact \ref{fact5}), it follows that $U$ satisfies 
	\begin{equation}\label{eq5}
	[au^{m} + \beta(u^m)b,u^{n_1},u^{n_2},\ldots,u^{n_k}]=0\end{equation}

We may assume $ a \notin C$ and $b \neq 0$ then $U$ satisfies non-trivial generalized polynomial identity. Therefore by (\cite{martindale1969pr},theorem 3), $U$ is dense subring of the ring of linear transformtion of a vector space $V$ over a division ring $D$. If $\beta$ is not Frobenius then from Fact (\ref{fac4.2}), $U$ satisfies 
\begin{equation}\label{eq6}
[au^{m} + z^m b,u^{n_1},u^{n_2},\ldots,u^{n_k}]=0
\end{equation} then by \cite{dhara2016eng}, we get, $a,b \in C$. In particular from equation $(\ref{eq6})$, $U$ satisfies
\begin{equation}\label{eqH}
b[z^m, u^{n_1}, u^{n_2},\ldots, u^{n_k}]=0
\end{equation}  then by posner's theorem \cite{posner1960pri} there exists a suitable filed $F$ and a positive integer $n$ such that $U$ and $M_{n}(F)$ satisfies the same polynomial identity. For $i \neq j$, choose $z= e_{ii} + e_{ij}$ and $u = e_{ii}$ in equation (\ref{eqH}), we get,\begin{equation*}
0=[e_{ii}+e_{ij},e_{ii}]_k = (-1)^k e_{ij}
\end{equation*} a contradiction, where $e_{ij}$ is the usual matrix unit with 1 at the $(i,j)^{th}$ entry and zero elsewhere.\\
Now we assume that $\beta$ is Frobenius and $dim_{D}V\geq 2$ (because if $dim_{D}V =1$ then $U$ will be commutative). If $char(R)=0$ then $\beta(x)= x$ because $\beta$ is Frobenius. This implies that $\beta$ is inner which is a contradiction. Thus we may assume that $char(R)=p\neq 0$ and $\beta(\lambda)= \lambda^{p^s}$ for all $\lambda \in C$ and for some fixed integer $s\geq 0$. Now replace $u$ by $\lambda u$ in equation $(\ref{eq5})$ then $U$ satisfies 
\begin{equation}\label{eq7}
\lambda^{m+n_1 +\ldots+n_k}	[au^{m} + \lambda^{m(p^s -1)}\beta(u^m)b,u^{n_1},u^{n_2},\ldots,u^{n_k}]=0.
\end{equation}
In particular choose $\lambda^{m}= \gamma$ then $U$ satisfies 
\begin{equation}\label{eq8}
[au^{m} + \gamma^{(p^s -1)}\beta(u^m)b,u^{n_1},u^{n_2},\ldots,u^{n_k}]=0.
\end{equation}
From equation $(\ref{eq5})$ and $(\ref{eq8})$ we get 
\begin{equation}\label{eq9}
(\gamma^{(p^s -1)}-1)[ \beta(u^m)b,u^{n_1},u^{n_2},\ldots,u^{n_k}]=0\
\end{equation} that is $U$ satisfies
\begin{equation}\label{eq10}
[ \beta(u^m)b,u^{n_1},u^{n_2},\ldots,u^{n_k}]=0.
\end{equation} Again from equation $(\ref{eq5})$ and $(\ref{eq10})$ we have that $[au^{m}, u^{n_1}, u^{n_2},\ldots, u^{n_k}]=0$ for all $u \in U$, this implies that $ a\in C$ ( see \cite{dhara2016eng}). Thus we may consider $ b \neq 0$.\\
Now let $ e^ 2 = e \in soc(R)$ then by equation $(\ref{eq10})$,
\begin{equation}\label{eqA}
  [\beta(e)b, e]_{k}=0.
\end{equation}
Right multiply above relation by $(1-e)$ gives
\begin{equation}\label{eq11}
e\beta(e)b(1-e) =0.
\end{equation}
In particular choosing the idempotent $(1-e)-(1-e)xe$ for all $x\in U$ in equation (\ref{eq11}), we have that $U$ satisfies:
\begin{equation}\label{eq12}
	(1-e)\beta(1-e)b(1-e)xe+ (1-e)\beta(1-e)\beta(x)\beta(e)b(e+(1-e)xe)
\end{equation}
\begin{equation*}-(1-e)xe\beta(1-e)b(e+(1-e)xe)+(1-e)xe\beta(1-e)\beta(x)\beta(e)b(e+(1-e)xe).
\end{equation*}
Since $\beta$ is outer then from  Fact (\ref{fac4.2}), $U$ satisfies :
\begin{equation}\label{eq13}
(1-e)\beta(1-e)b(1-e)xe+ (1-e)\beta(1-e)z\beta(e)b(e+(1-e)xe)
\end{equation}
\begin{equation*}-(1-e)xe\beta(1-e)b(e+(1-e)xe)+(1-e)xe\beta(1-e)z\beta(e)b(e+(1-e)xe).
\end{equation*}

In particular for $x=0$, we get $(1-e)\beta(1-e)z\beta(e)be=0$, for all $z\in U$. Hence by primeness of $U$, we have either $(1-e)\beta(1-e)=0$ or $\beta(e)be =0$ for any $e^2 =e \in Soc(U)$.\\
If we consider the case  $(1-e)\beta(1-e)=0$, then equation (\ref{eq12}) reduces to:
\begin{equation*}
		-(1-e)xe\beta(1-e)b(e+(1-e)xe)+(1-e)xe\beta(1-e)z\beta(e)b(e+(1-e)xe).
\end{equation*} In particular $U$ satisfies 
\begin{equation*}
	(1-e)xe\beta(1-e)z\beta(e)b(e+(1-e)xe).
\end{equation*} Now replace $z$ by $ze$ in above relation we get that
\begin{equation}\label{eq14}
(1-e)xe\beta(1-e)ze\beta(e)be.
\end{equation}
 Now since $e \beta(1-e)\neq 0$, otherwise from $(1-e)\beta(1-e)=0$ and we get $\beta(1-e)=0$, which is a contradiction. Thus from equation (\ref{eq14}), we get $e\beta(e)be =0$, then by equation (\ref{eq11}) we get that $e\beta(e)b=0$. Thus, in any case equation (\ref{eqA}) implies that 
 \begin{equation}\label{eq16}
 \beta(e)be =0.
 \end{equation} for any idempotent element $e \in U$.
 If we choose the idempotent element $ ex(1-e)+e$, for all $x \in U$ then from equation (\ref{eq16}), $U$ satisfies:
 \begin{equation*}
 \big(\beta\big(ex(1-e)+\beta(e)\big)b(ex(1-e)+e).
 \end{equation*}Since $\beta(e)be =0$, $U$ satisfies
\begin{equation*}
	\beta(e)\beta(x)b(ex(1-e)+e).
	\end{equation*}
Since $\beta$ is outer then from Fact (\ref{fac4.2}), $U$ satisfies:
\begin{equation*}
	\beta(e)yb(ex(1-e)+e)
\end{equation*} for all $x,y \in U$. In particular for $x=0$, $\beta(e)ybe=0$ i.e. $be =0$ (by primness of $U$) for all $e^2 =e \in U$. Let $M$ denotes the additive subgroup of $U$, which is generated by all the idempotent elements of $U$, then $bM=0$. Moreover, by \cite{herstein1969t}(page 18, corollary), $[U,U]\subseteq M$, i.e. $b[U,U]=0$ implying $b=0$, a contradiction.

\end{proof}
\textit{Remark:} Lemma \ref{lem1} and Lemma \ref{lem2} cover all the cases to prove Proposition \ref{prop1}.\\

\textbf{Proof of main theorem:} From the given hypothesis we have:
\begin{equation}\label{eq17}
	[F(u^m),u^{n_1},u^{n_2},\ldots,u^{n_k}]=0
\end{equation} for all $u \in I$. Since $R,I$ and $U$ satisfy the same generalized polynomial identities as well as the same differential identities with automorphism (see Fact \ref{fact4}, \ref{fact5}). Hence, 
\begin{equation}\label{eq18}
[F(u^m),u^{n_1},u^{n_2},\ldots,u^{n_k}]=0
\end{equation} for all $u \in U$, where $F(u) =cu+d(u)$ for some $c \in U$ and $d$ is skew derivation of $U$ with associated automorphism $\beta$ (see Fact \ref{fac3}).
 We divide the proof into the following cases:\\
 
\textbf{Case 1:} If $d$ is inner derivation then $d(u)= pu-\beta(u)p$ for all $u \in U$ and for some $p \in U$. Then $U$ satisfies 
\begin{equation}\label{eq19}
[(c+p)u^{m}-\beta(u^m)p,u^{n_1},u^{n_2},\ldots,u^{n_k}]=0.
\end{equation} Then by Proposition \ref{prop1} we get the required result.\\
\textbf{Case2:} If $d$ is outer then  $U$ satisfies 
\begin{equation}\label{eq20}
	\Big[cu^m + \sum_{i=0}^{m-1}\beta(u^i)d(u)u^{m-i-1},u^{n_1},u^{n_2},\ldots,u^{n_k}\Big]=0.
\end{equation} by \cite{Chuang2005}, we have:
\begin{equation}\label{eq21}
\Big[cu^m + \sum_{i=0}^{m-1}\beta(u^i)yu^{m-i-1},u^{n_1},u^{n_2},\ldots,u^{n_k}\Big]=0.
\end{equation}
In particular $U$ satisfies
\begin{equation}\label{eqB}
	\Big[\sum_{i=0}^{m-1}\beta(u^i)yu^{m-i-1},u^{n_1},u^{n_2},\ldots,u^{n_k}\Big]=0
\end{equation}
\textbf{Subcase 1:} If $\beta$ is outer derivation then from Fact (\ref{fac4.2}), $U$ satisfies:
\begin{equation}\label{eq22}
	\Big[cu^m + \sum_{i=0}^{m-1}z^iyu^{m-i-1},u^{n_1},u^{n_2},\ldots,u^{n_k}\Big]=0
\end{equation} for all $u,y,z \in U$. In particular for $z=0$, we have that 
\begin{equation}\label{eq23}
	[yu^{m-1},u^{n_1},u^{n_2},\ldots,u^{n_k}]=0.
\end{equation} Then by Posner's theorem, there exists a suitable field $F$ and a positive integer $n$ such that $U$ and $M_n(F)$ satisfy the same generalized polynomial identity. For $ i\neq j$, choose $ y = e_{ii} +e_{ji}$ and $u = e_{ii}$ in equation (\ref{eq23}), we get 
\begin{equation*}
	0= [e_{ii} +e_{ji}, e_{ii}]_k =e_{ji}
\end{equation*} which is a contradiction.

\textbf{Subcase 2:} If $\beta$ is inner then $\beta(x)= pxp^{-1}$, for all $x \in U$ and for some $p \in U$. Hence from equation (\ref{eqB}), $U$ satisfies:
\begin{equation}\label{eq24}
	\Big[\sum_{i=0}^{m-1}pu^ip^{-1}yu^{m-i-1},u^{n_1},u^{n_2},\ldots,u^{n_k}\Big]=0.
\end{equation}
Since $\beta$ is not identity, hence $ p\notin C$,  therefore equation $(\ref{eq24})$ is a non-trivial polynomial identity for $R$. By  \cite{martindale1969pr}, $U$ is isomorphic to a dense ring of linear transformation on some vector space $V$ over $C$.
Firstly, we consider the case when $dim_C V\geq 3$. Since $p\notin C$ therefore there exists some $v \in V$ such that $\{p^{-1}v, v\}$ is linearly $C$-independent. Since $dim_C V \geq 3$, there must exists $ w_1 \in V$ such that $\{p^{-1}v,v,w_1\}$ is linearly $C$- independent. Since $U$ is dense, therefore by Jacobson density theorem ( see Fact \ref{fact7}), there exists $y_1,y_2 \in U$ such that 
\begin{equation*}
	y_1w_1 =0, y_2w_1 =v, y_1p^{-1}v =p^{-1}v, y_1v =v.
\end{equation*}
Therefore by equation (\ref{eq24}) we get that \begin{equation*}
0= \Big[\sum_{i=0}^{m-1}py_1^ip^{-1}y_2y_1^{m-i-1},y_1^{n_1},y_1^{n_2},\ldots,y_1^{n_k}\Big]w_1 = (-1)^{k}v \neq 0
\end{equation*} which is a contradiction.
Now if $dim_C V = 2$, i.e. $U \cong M_2(C)$, ring of $2$-order matrix over field $C$.
\begin{equation*}
p = \begin{bmatrix}
	p_{11} & p_{12}\\
	p_{21}& p_{22}\\
\end{bmatrix},
p^{-1}= \frac{1}{det(p)}\begin{bmatrix}
	 p_{22} & -p_{12}\\
	-p_{21} &  p_{11}\\
\end{bmatrix}
\end{equation*} choosing $u = e_{22}, y = e_{21}$ in equation $(\ref{eq24})$, we get that
\begin{equation}\label{eqE}
	p_{11} p_{22} =0
\end{equation}
Similarly, by choosing $u= e_{11}, y= e_{22}$ in equation $(\ref{eq24})$, we get that
\begin{equation}\label{eqF}
p_{11}p_{12}=0
\end{equation}Since $p$ is invertible therefore $p_{22}$ and $p_{12}$ can not be zero simultaneously, thus from equation (\ref{eqE}) and (\ref{eqF}), it follows $p_{11}=0$. This implies $p_{12} \neq 0$, otherwise $p$ will be singular matrix.\\
Choose $\phi(u)= (1-e_{12})u(1+e_{12}) \in Aut(U)$, then 
\begin{equation*}
\phi\Bigg(\Big[\sum_{i=0}^{m-1}pu^ip^{-1}yu^{m-i-1},u^{n_1},u^{n_2},\ldots,u^{n_k}\Big]\Bigg)=0\end{equation*} i.e. $U$ satisfies:
\begin{equation}
\Big[\sum_{i=0}^{m-1}\phi(p)u^i\phi(p^{-1})yu^{m-i-1},u^{n_1},u^{n_2},\ldots,u^{n_k}\Big]=0.
\end{equation}Denote $\phi(p)_{11}$ as the $(1,1)^{th}$-entry of $\phi(p)$, then by the same argument as above we get $ 0= \phi(p)_{11}= p_{21}$, which is a contradiction (because $p$ is invertible.).\\
For the case $dim_C V=1$, $R$ will be commutative and we have nothing to prove in this case.\\

Following is the very natural consequence of our main theorem:
\begin{cor}
Let $R$ be a prime ring with its Utumi ring of quotients $U$, $F$ a nonzero generalized
derivation of $R$ and $I$ a non-zero ideal of $R$. Suppose that $[F(x^{m}), x^{n}]_k = 0$ for all $x \in I$,where $n,k \geq 1$ are fixed integers. Then there exists $\beta \in C$ such that $F(x) = \beta x$ for all $x \in R$.
\end{cor}
\begin{proof}
Choosing $n_1 =n_2 =\ldots=n_k =n$ in our main theorem we get the required result.
\end{proof}

\textbf{Future research:} Recently, C.K. Liu in \cite{liu2021engel} investigated the structure of $F$ and $G$ if they satisfy the identity $[F(x^n)x^m+x^mG(x^n),x^r]_k =0$ for all $x\in I$, where $F,G$ are non-zero generalized derivation on a prime ring $R$, $I$ is the non-zero ideal of $R$ and $m,n,k \geq 1$ are fixed integers. In the light of \cite{liu2021engel} with this article one can try to find the structure of $F$ and $G$ if they satisfy the identity $[F(x^n)x^m+x^mG(x^n),x^{n_1},x^{n_2},\ldots,x^{n_k}]=0$ for all $x\in I$, where $m,n,n_1,n_2,\ldots,n_k\geq 1$ are fixed integers.

\section*{Acknowledgement}
The authors is highly thankful to the referee(s) for valuable suggestions and comments. This research is not funded.


\begin{thebibliography}{99}
	\bibitem{dhara2020en}
	Dhara, Basudeb and De Filippis, Vincenzo,
	Engel conditions of generalized derivations on left ideals and Lie ideals in prime rings, Communications in Algebra, 48 (1), 154--167, 2020.
	\bibitem{dhara2011}
Dhara, Basudeb,
	Annihilator condition on power values of derivations, Indian Journal of Pure and Applied Mathematics, 42 (4), 357--369, 2011.
	\bibitem{albacs2008}
	Alba{\c{s}}, Emine and Arga{\c{c}}, Nurcan and Filippis, Vincenzo De, 
	Generalized derivations with Engel conditions on one-sided ideals, 36 (6), 2063--2071, 2008.
	
\bibitem{de2012}
	De Filippis, Vincenzo and Di Vincenzo, Onofrio Mario,
	Vanishing derivations and centralizers of generalized derivations on multilinear polynomials, Communications in Algebra, 40 (6), 1918--1932,  (2012).
	\bibitem{beidar1995}
Beidar, Konstant I and Martindale, Wallace S and Mikhalev, Alexander V,
	Rings with generalized identities, CRC Press, 1995.
\bibitem{argac2008}
Argac, Nurc and Carini, Luisa and De Filippis, V,
An Engel condition with generalized derivations on Lie ideals, {Taiwanese Journal of Mathematics},  {419--433}, 2008.
\bibitem{alahmadi}
Alahmadi, Adel and Ali, Shakir and Khan, Abdul Nadim and Khan, Mohammad Salahuddin,
A characterization of generalized derivations on prime rings, {Communications in Algebra}, 44 (8) {3201--3210}, 2016.
\bibitem{kharchenko1978diff}	
	Kharchenko, Vladislav Kirillovich,
	Differential identities of prime rings, Algebra and Logic, 17 (2), 155--168, 1978.
	\bibitem{dhara2016eng}
	Dhara, Basudeb and Ali, Asma and Das, Deepankar, 
	Engel conditions of generalized derivations on Lie ideals and left sided ideals in prime rings and Banach Algebras, Afrika Matematika, 27 (7), 1391--1401, 2016.

\bibitem{beidar1978ri}
Beidar, KI,  
  {Rings with generalized identities. 3.}, {Vestnik Moskovskogo Universiteta Seriya i Matematika, Mekhanika},  4, 66--73, 1978.
\bibitem{chuang1988gpis}
  Chuang, Chen Lian,
  GPIs having coefficients in Utumi quotient rings, Proceedings of the American Mathematical Society, 103 (3), 723--728, 1988.

	\bibitem{demir2010re}
	Demir, {\c{C}}agri and Arga{\c{c}}, Nurcan,
	A result on generalized derivations with Engel conditions on one-sided ideals, Journal of the Korean Mathematical Society, 47 (3), 483--494, 2010.

\bibitem{posner1957}
  Posner, Edward C,
  {Derivations in prime rings}, {Proceedings of the American Mathematical Society}, 8 (6), 1093--1100, 1957.
	
	\bibitem{martindale1969pr}
Martindale 3rd, Wallace S,
	Prime rings satisfying a generalized polynomial identity, Journal of Algebra, 12 (4), 576--584, 1969.
	\bibitem{erickson1975pr}
	Erickson, Theodore and Martindale, Wallace and Osborn, James,
	Prime nonassociative algebras, Pacific Journal of Mathematics, 60 (1), 49--63 1975.
	\bibitem{faith1963new}
	Faith, C and Utumi, Y,
	On a new proof of Litoff's theorem, {Acta Mathematica Academiae Scientiarum Hungarica},  14 (3-4), 369--371, 1963.
	\bibitem{lee1999}
	Lee, Tsiu Kwen,
	Generalized derivations of left faithful rings, Communications in Algebra, 27 (8), 4057--4073, 1999. 

\bibitem{herstein1969t}
 Herstein, Israel Nathan,
  {Topics in ring theory}, University of Chicago press, 1969.

\bibitem{di1989n}
Di Vincenzo, OM,
{On the n-th centralizer of a Lie ideal}, Boll. UMI, 7 (3-A), 77--85, 1989.
 
\bibitem{lanski1972li}
  Lanski, Charles and Montgomery, M Susan,
  {Lie structure of prime rings of characteristic 2}, Pacific Journal of Mathematics, 42 (1), 117--136, 1972.
  
  
\bibitem{posner1960pri}
Posner, Edward C,  
  {Prime rings satisfying a polynomial identity}, Proceedings of the American Mathematical Society, 11 (2), 180--183, 1960.
  
  \bibitem{lee1999result}
Lee, Tsiu-Kwen and Shiue, Wen-Kwei,  
  {A result on derivations with Engel condition in prime rings}, Southeast Asian Bulletin of Mathematics,
 23 (3), 437--446, 1999.
 
 \bibitem{lee1992semiprime}
Lee, Tsiu Kwen,	
{Semiprime rings with differential identities}, Bulletin of the Institute of Mathematics, 20 (1), 27--38, 1992.


\bibitem{liu2021engel}
Liu, Cheng-Kai,
	{An Engel condition with two generalized derivations in prime rings}, Communications in Algebra, 49 (2), 836--849, 2021.

\bibitem{lanski1993engel}
  Lanski, Charles,
  {An Engel condition with derivation}, {Proceedings of the American Mathematical Society}, 118 (3), 1993.
  


\bibitem{brevsar1993centralizing}
Bre{\v{s}}ar, Matej,	
	{Centralizing mappings and derivations in prime rings}, Journal of Algebra, 156 (2), 385--394, 1993.

\bibitem{tiwari2021identities}
Tiwari, SK,  
  {Identities with generalized derivations in prime rings}, Rendiconti del Circolo Matematico di Palermo Series 2, 1--17, 2021.
  
  \bibitem{jacobson1956structure}
Jacobson, Nathan, American Mathematical Society, 37, 1956.
Structure of rings, 	
	
\bibitem{carini2016identities}
Carini, Luisa and De Filippis, Vincenzo and Scudo, Giovanni,	
	{Identities with product of generalized skew derivations on multilinear polynomials}, Communications in Algebra, 44 (7), 3122--3138, 2016.
	
	\bibitem{de2019generalized}
De Filippis, Vincenzo and Di Vincenzo, Onofrio Mario,	
	{Generalized Skew Derivations and Nilpotent Values on Lie Ideals}, Algebra Colloquium, 26 (04), 2019.
\bibitem{chuang1993differential} C. L. Chuang. {Differential identities with automorphisms and antiautomorphisms, ii. Journal
of Algebra},160(1):130–171, 1993.
\bibitem{Kharchenko1991}V. Kharchenko.{Automorphisms and Derivations of Associative Rings}, volume 69. Springer
Science and Business Media, 1991.
\bibitem{Chuang2005}C.-L. Chuang and T.-K.{Identities with a single skew derivation},Journal of Algebra,
288(1):59 – 77, 2005.
\end{thebibliography}
\end{document}